\documentclass[12pt,oneside,english]{article}
\usepackage[T1]{fontenc}
\usepackage[latin9]{inputenc}
\usepackage{amstext}
\usepackage{amsthm}
\usepackage{amssymb}
\usepackage{graphicx,color,fancyhdr,float}

\usepackage{tikz}
\usetikzlibrary{arrows}
\usetikzlibrary{ positioning, decorations.markings,decorations.pathreplacing,calligraphy}
\usetikzlibrary{shapes.geometric}
\tikzset{
vtx/.style={inner sep=2pt, outer sep=0pt, circle, fill=black,draw=black}
}

\usepackage{comment}
\usepackage{amsmath,amssymb,amsthm,amsfonts,amsbsy,latexsym,dsfont}
\usepackage{mathrsfs}
\makeatletter
\numberwithin{equation}{section}
\numberwithin{figure}{section}
\usepackage{enumitem}
\setlist{nosep}
\setlist[enumerate,1]{label = (\alph*),
                      ref   = \alph*}
\theoremstyle{plain}
\newtheorem{thm}{\protect\theoremname}
\theoremstyle{plain}
\newtheorem{conjecture}[thm]{\protect\conjecturename}
\theoremstyle{plain}
\newtheorem{corollary}[thm]{\protect\corollaryname}
\theoremstyle{plain}
\newtheorem{lem}[thm]{\protect\lemmaname}
\newlist{Oinline}{enumerate*}{1}
\setlist[Oinline]{label=(O\arabic*)}
\newlist{romaninline}{enumerate*}{1}
\setlist[romaninline]{label=(\roman*)}
\newlist{alphainline}{enumerate*}{1}
\setlist[alphainline]{label=(\alph*)}
\newlist{Cenum}{enumerate}{1}
\setlist[Cenum]{label=(C\arabic*)}

\usepackage{fullpage}
\DeclareMathOperator{\cl}{cl}

\makeatother

\usepackage{babel}

\providecommand{\conjecturename}{Conjecture}
\providecommand{\lemmaname}{Lemma}

\providecommand{\theoremname}{Theorem}
\providecommand{\corollaryname}{Corollary}
\def\wt{\widetilde}
\def\wh{\widehat}
\def\BB{{\mathcal B}}
\def\mm{~\textrm{mod}^*~}

\title{Equitable list coloring of 
 planar graphs with given maximum degree}

\author{
{{H. A. Kierstead}}\thanks{
\footnotesize {Arizona State University
Tempe, AZ, USA. E-mail: \texttt {kierstead@asu.edu}.
}}
\and
{{Alexandr Kostochka}}\thanks{
\footnotesize {University of Illinois at Urbana--Champaign, Urbana, IL 61801. E-mail: \texttt {kostochk@illinois.edu}.
 Research 
is supported in part by  NSF  Grant DMS-2153507 and by NSF RTG Grant DMS-1937241.
}}
\and
{{Zimu Xiang}}\thanks{University of Illinois at Urbana--Champaign, Urbana, IL 61801. E-mail: {\tt zimux2@illinois.edu}.
 Research 
is supported in part by  NSF RTG Grant DMS-1937241.}}

\date{\today}

\begin{document}
\maketitle

\begin{abstract}
If $L$ is a list assignment of $r$ colors to each vertex of an $n$-vertex graph 
$G$, then an {\em equitable $L$-coloring} of $G$ is a proper  coloring of 
vertices of $G$ from their lists such that no color is used more than
 $\lceil n/r\rceil$ times. 
A graph is \emph{equitably} $r$-\emph{choosable} 
if it has an equitable $L$-coloring for every $r$-list assignment $L$.
 In 2003, Kostochka, Pelsmajer and West (KPW)
 conjectured that an analog of the famous Hajnal-Szemer\' edi Theorem on equitable coloring holds for equitable list coloring, namely, that for each positive integer $r$  every graph $G$ with maximum degree at most $r-1$  is {equitably} $r$-{choosable}.

The main result of this paper is that for each $r\geq 9$ and each planar graph $G$, a stronger statement holds:
if the maximum degree of $G$ is at most $r$, then $G$  is {equitably} $r$-{choosable}. In fact, we prove the result for a broader class of graphs---the class ${\mathcal{B}}$ of  the graphs in which each bipartite subgraph $B$ with $|V(B)|\ge3$ has at most $2|V(B)|-4$ edges.
Together with some known results, this implies that the KPW Conjecture holds for all graphs in ${\mathcal{B}}$, in particular, for all
planar graphs. 
\end{abstract}



\section{Introduction}
For a graph $G$, 
$V(G)$ denotes the set of its vertices, $E(G)$ --- the set of its edges, $\Delta(G)$ --- its maximum degree, and $\delta(G)$ --- its minimum degree.
We let $|G|=|V(G)|$ and $\|G\|=|E(G)|$. 
For disjoint $X,Y\subseteq V(G)$,  $G[X]$ denotes the subgraph of $G$ induced by $X$, and $\|X,Y\|=\|X,Y\|_G$ ---
the number of edges in $G$ connecting $X$ with $Y$. If $X=\{x\}$ we may simply write $\|x,Y\|$ for $\|X,Y\|$. The degree of $x$ in $G$ is noted by $d(x)=d_{G}(x):=\|x,V(G)-x\|_{G}$. 

Let $G$ be a graph. An {\em equitable $r$-coloring} of $G$ is a proper coloring with  $r$ colors such
that any two color classes differ in size by at most one. The following is an equivalent, but more awkward, way of stating
this whose utility will become apparent later. Let $n\mm r$ be the unique integer $i$ 
in $[r]$ with $n-i$ divisible by $r$. 
\begin{align}
 \label{SE}\textrm{All classes have size at most } \lceil \frac{n}{r}\rceil, \textrm{ and at most } n \mm{r} \textrm{ classes have size } \lceil \frac{n}{r}\rceil.   
 \end{align}

Erd\H{o}s
\cite{Er} conjectured, and Hajnal and Szemer\' edi  \cite{HS} proved that if $\Delta(G)+1\leq r$
then $G$ has an equitable $r$-coloring. Chen, Lih and Wu~\cite{CLW} presented a Brooks-type
version of Erd\H{o}s' conjecture.
\begin{conjecture}[Chen, Lih \& Wu]
\label{conj:CLW}If $G$ is an $r$-colorable graph with $\Delta(G)\leq r$ then
either $G$ has an $r$-equitable coloring or both $r$ is odd and $K_{r,r}\subseteq G$.
\end{conjecture}

This conjecture was proved for some classes of graphs.  Lih and Wu~\cite{lih1996equitable} proved the conjecture for bipartite graphs. Chen, Lih and Wu~\cite{CLW} themselves proved the conjecture for $r=3$ and for $r\geq |V(G)|/2$. 
Kierstead and Kostochka proved the conjecture in~\cite{KK4} for $r=4$ and  in \cite{KK} for $r\ge |V(G)|/4$. 
Yap and Zhang~\cite{YZ} proved that the conjecture holds for planar graphs when $r\geq 13$, Nakprasit~\cite{KN,N9} confirmed the conjecture for planar graphs when $9\leq r\leq 12$, and
 Kostochka, Lin and Xiang~\cite{KLX} proved it for planar graphs when $r=8$. Together, these results can be stated
as follows.
\begin{thm}[\cite{YZ} for $r\geq 13$,~\cite{KN,N9} for $9\leq r\leq 12$ and~\cite{KLX} for $r=8$]\label{YZNKLX}
If $r\geq 8$ and
$G$ is a planar graph with $\Delta(G)\leq r$, then  $G$ has an equitable $r$-coloring.
\end{thm}

There also was a series of papers~\cite{li2009equitable,kknak2012equitable,zhu2008equitable}  on equitable coloring of planar graphs with restrictions on cycle structure.

In this paper we investigate the list-coloring
version of Conjecture~\ref{conj:CLW} for planar graphs. An $r$-\emph{list assignment}
for $G$ is a function $L$ that assigns a set $L(v)$ of $r$ available colors to
each vertex $v\in V$. Fix an $r$-list assignment $L$. 
An $L$-\emph{coloring $f$}
of $G$ is a proper coloring of $G$ such that $f(v)\in L(v)$ for each vertex 
$v\in V$. Kostochka, Pelsmajer and West~\cite{KPW}  defined 
an $L$-coloring $f$ of $G$ to be \emph{equitable} if the size of every color class is at
most $\lceil|G|/r\rceil$, and called a graph \emph{equitably} $r$-\emph{choosable}
if it had an equitable $L$-coloring for every $r$-list assignment $L$. Notice
that in the list setting it is unrealistic to require that any two color classes
differ in size by at most one---the list of some vertex might only contain colors
that are in no other lists, while all other lists could be identical. It is possible that $L$  assigns all vertices the same list; in this case  $L$  is \emph{plain}. An equitable $L$-coloring might not be equitable, even when  $L$  is plain. For example, if $r=3$, 
$G=C_4$   and two color classes each have size $2$. 

Kostochka, Pelsmajer and West conjectured  the following analog of the  Hajnal--Szemer\' edi Theorem.
\begin{conjecture}
\label{conj:KPW} For each $r$, every graph $G$ with $\Delta(G)\leq r$ is equitably $(r + 1)$-choosable.
\end{conjecture}

Pelsmajer [10] and independently Wang and
Lih [13] proved Conjecture~\ref{conj:KPW} for $r=3$. Kierstead and Kostochka~\cite{KK3} proved Conjecture~\ref{conj:KPW} for $r\leq 7$.
In this paper we prove for planar graphs with maximum degree at least $9$ the following list analogue of Conjecture~\ref{conj:CLW}
which 
 strengthens  Conjecture~\ref{conj:KPW} for such graphs.

\begin{thm}
\label{thm:Main} Every planar graph $G$ is equitably $r$-choosable, if $r\geq\max\{9,\Delta(G)\}$. 
\end{thm}

In fact, we will show that the claim of  Theorem~\ref{thm:Main} holds 
for a   more general class of graphs than planar graphs. 
Let ${\mathcal{B}}$ denote the class  of graphs such that
\begin{equation}
\|X,Y\|\leq2|X\cup Y|-4~~\text{for all disjoint}~~X,Y\subseteq V(G)~~\text{with}~~|X\cup Y|\ge3.\label{eq:bibound}
\end{equation}

By Euler's Formula, each planar graph $G$ is in ${\mathcal{B}}$.
Hence Theorem~\ref{thm:Main} follows from:

\begin{thm}
\label{thm:Main2-} Every graph $G\in \BB$ 
  is equitably $r$-choosable, if $r\geq\max\{9,\Delta(G)\}$. 
\end{thm}

Since Conjecture~\ref{conj:KPW} for $r\leq 7$ was proved in~\cite{KK3}, Theorem~\ref{thm:Main2-} yields the following fact.

\begin{corollary}
\label{cor:Main} For each $r$, every  graph $G\in \BB$
 with $\Delta(G)\leq r$ is equitably $(r + 1)$-choosable.
\end{corollary}

Our proof of Theorem~\ref{thm:Main2} yields a somewhat stronger result. In a recent related project, we needed the following stronger (and arguably more natural) version of equitable $L$-coloring in order to make an inductive argument work. Define an $L$-coloring to be \emph{strongly equitable (SE)} if it satisfies \eqref{SE}, and call a graph $G$ \emph{strongly equitably (SE) $r$-choosable} if it has an SE $L$-coloring for every $r$-list assignment $L$. Notice that if $L$ is a plain $r$-list assignment for $G$ then every SE $L$-coloring of $G$ is an equitable coloring of $G$, and recall that this is not true for equitable $L$-colorings. Now Theorem~\ref{thm:Main2-} will follow from our strongest theorem:

\begin{thm}
\label{thm:Main2} Every graph $G\in \BB$ 
  is SE $r$-choosable,  if $r\geq\max\{9,\Delta(G)\}$. 
\end{thm}

The proof of Theorem~\ref{thm:Main2} is essentially constructive.
One can naturally extract from it a polynomial-time algorithm that produces for every
 graph $G\in \BB$  with $\Delta(G)\leq r$ and every $(r + 1)$-list $L$
 for $G$ an 
 SE $L$-coloring of $G$.

\bigskip
\noindent\textbf{Notation.} 
Let $D=(\Gamma,A)$ be a digraph with node set $\Gamma$ and arc set $A$. 
In the context of digraphs, an $\alpha,\gamma$-\emph{path} is a directed path from $\alpha$ to $\gamma$. 
Vertex $\gamma$ is \emph{reachable}
from vertex $\alpha$ and $\alpha$ is \emph{accessible} from $\gamma$, if there
is a directed $\alpha,\gamma$-path in $D$. For $\Theta\subseteq \Gamma$, let $\cl^{+}(\Theta)$
denote the set of $\gamma\in \Gamma$ that are reachable from a vertex of $\Theta$. A set
$\Theta\subseteq \Gamma$ is a \emph{sink} if $\cl^{+}(\Theta)=\Theta$.

For a positive integer  $i$, put $[i]=\{1,2,\ldots,i\}$.
 If $S$ is a subset of $T$, and  $x\in T$
 then put $S+x:=S\cup\{x\}$ and $S-x:=S\smallsetminus\{x\}$.

\section{Set-up for the proof of Theorem~\ref{thm:Main2}}

We argue by induction on $|G|$. Fix integer $r\ge 9$, graph $G=(V,E)\in\BB$ with $\Delta(G)\le r$ and $r$-list assignment $L$ for $G$. The case $|G|=1$ is trivial, so suppose $|G|>1$.

Fix $p\in V$ with $d(p)=\delta(G)$. Set $V':=V-p$ and $G':=G[V']$. 
By induction, graph $G'$ has an SE $L$-coloring $f:V'\to \Gamma$. Now $f$ defines an equivalence relation $\sim$
on $V'$ by $x\sim y$ iff $f(x)=f(y)$. 
Let $s:=\lceil|G'|/r\rceil$ and $\Gamma:=\bigcup_{x\in V}L(x)$.
Denote the equivalence class of $x$ by
$\widetilde{x}=\widetilde{x}(f)$. If $f(x)=\gamma$ then $f^{-1}(\gamma)=\widetilde{x}$.
We 
extend this notation to $\gamma\in \Gamma$ by setting $\widetilde{\gamma}=\widetilde{\gamma}(f):=f^{-1}(x)=\widetilde{x}$.
It is possible that $\widetilde{\gamma}=\emptyset$. Call a class $\widetilde{\gamma}$
 \emph{light} if $|\widetilde{\gamma}|<s$, \emph{full} if $|\widetilde{\gamma}|=s$
and \emph{overfull} if $|\widetilde{\gamma}|>s$. Using this language, $f$ being SE means that no class is overfull and at most $|G|\mm r$ classes are full.

Let $H=(\Gamma,F)=H(f)$ be an auxiliary digraph on the color set $\Gamma$ with arc set
$F$, where $\alpha\beta\in F$ if there is a vertex $x\in\widetilde{\alpha}$ such
that $\beta\in L(x)$ and $N(x)\cap\widetilde{\beta}=\emptyset$. In this case, $x$
is \emph{movable} to $\widetilde{\beta}$ and $x$ is a \emph{witness} for the arc
$\alpha\beta$. If $\beta\in L(x)$ and $N(x)\cap\widetilde{\beta}=\{y\}$ then $y$
\emph{blocks} $x$. Notice that if $x$ is a witness for $\alpha\beta$ then we can
obtain a new (possibly inequitable) $L$-coloring by moving $x$ 
from $\widetilde{\alpha}$
to $\widetilde{\beta}$. This operation may change $H$ in several ways. First, $x$
is no longer available to witness outedges of $\alpha$, but it may witness outedges
of $\beta$. Also, $x$ no longer blocks inedges of $\alpha$, but it may block inedges
of $\beta$.

For a subset $\Theta\subseteq \Gamma$, put $\widehat{\Theta}=\widehat{\Theta}(f):=\{\widetilde{\gamma}:\gamma\in \Theta\}$\footnote{The empty set is allowed to appear multiple times in $\widehat{\Theta}$, once for every unused color in $\Theta$.}
and $\widetilde{\Theta}=\widetilde{\Theta}(f):=\bigcup\widehat{\Theta}$. Let $\Lambda_{0}=\Lambda_{0}(f):=\{\gamma\in \Gamma:|\widetilde{\gamma}|<s\}$.
So $\widehat{\Lambda}_{0}$ is the set of light classes. As $r(s-1)\leq|G'|<|G|\leq rs$,
$\Lambda_{0}\ne\emptyset$. Call a color $\gamma$, and a class $\wt\gamma$, \emph{accessible}, if there is a (directed)
$\gamma,\Lambda_{0}$-path $P$ (possibly trivial) in $H$. We say that $P$ \emph{witnesses}
that $\gamma$ is accessible. Let $\Lambda=\Lambda(f)$ be the set of accessible colors, and
put $\Phi=\Phi(f):=\Gamma\smallsetminus \Lambda(f)$. Now every class in $\widehat{\Phi}$ is full. Put 
$B:=\widetilde{\Phi}$,
$b:=|B|$, $a:=r-b$ and $A:=\widetilde{\Lambda}$ (See Figure~\ref{fig:LP}). 
Then $|B|=bs$ and $(a-1)s\leq|A|\leq as-1$.
If $|\Lambda|=a$ then $|\Gamma|=r$, so $L(v)=\Gamma$ for every vertex $v\in V$, i.e., L  is plain. Thus
\begin{equation}
\mbox{ $|\Lambda|>a$ or $L$ is plain.}\label{eq:|A|>s}
\end{equation}

\begin{figure}[t] 
    \begin{minipage}[t]{0.55\textwidth}
    \centering
         \begin{tikzpicture}
    
    \draw (0,-.1) rectangle (7.5,.7);
    \draw (-0.3,0.25) node(H) {$H$};
    
    \draw (0,-0.6) rectangle (.75,-0.2);
    \fill[red!10!white] (0,-0.6) rectangle (.75,-0.2);
    \node at (.375,-0.9) {$\widetilde\gamma_1$};
    \draw (.75,-1) rectangle (1.5,-0.2);
    \fill[red!10!white] (.75,-1) rectangle (1.5,-0.2);
    \node at (1.125,-1.3) {$\widetilde\gamma_2$};
    \draw (1.5,-1.5) rectangle (2.25,-0.2);
    \fill[red!10!white] (1.5,-1.5) rectangle (2.25,-0.2);
    \node at (1.875,-1.8) {$\widetilde\gamma_3$};
    \draw (2.25,-2) rectangle (3,-0.2);
    \fill[red!10!white] (2.25,-2) rectangle (3,-0.2);
    \node at (2.625,-2.3) {$\widetilde\gamma_4$};
    \draw (3,-2) rectangle (3.75,-0.2);
    \fill[blue!10!white] (3,-2) rectangle (3.75,-0.2);
    \node at (3.375,-2.3) {$\widetilde\gamma_5$};
    \draw (3.75,-2) rectangle (4.5,-0.2);
    \fill[blue!10!white] (3.75,-2) rectangle (4.5,-0.2);
    \node at (4.125,-2.3) {$\widetilde\gamma_6$};
    \draw (4.5,-2) rectangle (5.25,-0.2);
    \fill[blue!10!white] (4.5,-2) rectangle (5.25,-0.2);
    \node at (4.875,-2.3) {$\widetilde\gamma_7$};
    \draw (5.25,-2) rectangle (6,-0.2);
    \fill[blue!10!white] (5.25,-2) rectangle (6,-0.2);
    \node at (5.625,-2.3) {$\widetilde\gamma_8$};
    \draw (6,-2) rectangle (6.75,-0.2);
    \fill[blue!10!white] (6,-2) rectangle (6.75,-0.2);
    \node at (6.375,-2.3) {$\widetilde\gamma_9$};
    \draw (6.75,-2) rectangle (7.5,-0.2);
    \fill[blue!10!white] (6.75,-2) rectangle (7.5,-0.2);
    \node at (7.125,-2.3) {$\widetilde\gamma_{10}$};

    \draw (-0.3,-.6) node(G) {$G'$};
    
    \draw (.375, 0.4) node[vtx,scale=0.5,color=red] (){};
    \node at (.375,0.1) {$\gamma_1$};
    \draw (1.125, 0.4) node[vtx,scale=0.5,color=red] (){};
    \node at (1.125,0.1) {$\gamma_2$};
    \draw (1.875, 0.4) node[vtx,scale=0.5,color=red] (){};
    \node at (1.875,0.1) {$\gamma_3$};
    \draw (2.625, 0.4) node[vtx,scale=0.5,color=red] (){};
    \node at (2.625,0.1) {$\gamma_4$};
    \draw (3.375, 0.4) node[vtx,scale=0.5,color=blue] (){};
    \node at (3.375,0.1) {$\gamma_5$};
    \draw (4.125, 0.4) node[vtx,scale=0.5,color=blue] (){};
    \node at (4.125,0.1) {$\gamma_6$};
    \draw (4.875, 0.4) node[vtx,scale=0.5,color=blue] (){};
    \node at (4.875,0.1) {$\gamma_7$};
    \draw (5.625, 00.4) node[vtx,scale=0.5,color=blue] (){};
    \node at (5.625,0.1) {$\gamma_8$};
    \draw (6.375, 0.4) node[vtx,scale=0.5,color=blue] (){};
    \node at (6.375,0.1) {$\gamma_9$};
    \draw (7.125, 0.4) node[vtx,scale=0.5,color=blue] (){};
    \node at (7.125,0.1) {$\gamma_{10}$};

    \draw [decorate,
    decoration = {calligraphic brace,
        raise=5pt,
        aspect=0.5,
        amplitude=7pt}] (0,.6) -- (3,.6);
    
    \draw (1.5,1.3) node(Lambda) {$\Lambda$};
    
    \draw [decorate,
    decoration = {calligraphic brace,
        raise=5pt,
        aspect=0.5,
        amplitude=7pt}] (3,.6) -- (7.5,.6);
    
    \draw (5.25,1.3) node(Phi) {$\Phi$};
    
    \draw [decorate,
    decoration = {calligraphic brace,
        raise=5pt,
        aspect=0.5,
        amplitude=7pt}] (3,-2.35) -- (0,-2.35);
    
    \draw (1.5,-3.15) node(A) {$A:=\widetilde{\Lambda}:=\bigcup\widehat{\Lambda}$};
    
    \draw [decorate,
    decoration = {calligraphic brace,
        raise=5pt,
        aspect=0.5,
        amplitude=7pt}] (7.5,-2.35) -- (3,-2.35);
    
    \draw (5.25,-3.15) node(B) {$B:=\widetilde{\Phi}:=\bigcup\widehat{\Phi}$};

     \draw [decorate,
    decoration = {lineto,
        raise=5pt,
        aspect=0.5,
        amplitude=7pt}] (7.6,-0.2) -- (7.8,-0.2);

    \draw [decorate,
    decoration = {lineto,
        raise=5pt,
        aspect=0.5,
        amplitude=7pt}] (7.6,-2) -- (7.8,-2);

    \draw [decorate,
    decoration = {lineto,
        raise=5pt,
        aspect=0.5,
        amplitude=7pt}] (7.7,-0.2) -- (7.7,-0.9);

    \draw [decorate,
    decoration = {lineto,
        raise=5pt,
        aspect=0.5,
        amplitude=7pt}] (7.7,-2) -- (7.7,-1.3);
    
    \draw (7.7,-1.1) node(s) {$s$};
    \end{tikzpicture}
    \caption{Coloring $f$ of $G'$ and its auxilary\\ graph $H$: $b=6$ and if $r=9$ then $a=3$}
    \label{fig:LP}
    \end{minipage}
    \begin{minipage}[t]{0.45\textwidth}
    \centering
        \begin{tikzpicture}
        [decoration={markings, 
        mark= at position 0.5 with {\arrow{stealth}}}
    ] 
    
        \draw (0,-.1) rectangle (5,.5);
        \draw (-.3,0.2) node(P) {$P$};
        \draw (0.5,0.2) node[vtx, scale=0.75] (n0){};
        \foreach \x in {1,...,4} {
       \draw (\x+0.5, 0.2) node[vtx,scale=0.75] (n\x){};
       \draw [postaction={decorate}] (\x+0.5,0.2) -- (\x-0.5,0.2);
    }
    
    \foreach \x in {0,...,4} {
       \draw (\x,-1.5) rectangle (\x+1,-0.2);
    }
    \draw (-0.3,-.85) node(Shifting) {$G$};
    
    \draw (4.5, -.45) node[vtx,scale=0.75] (v1){};
        \draw (4.5, -.45) node[yshift=-0.3cm] {$v_1$};
    \draw (3.5, -1.25) node[vtx,scale=0.75] (v2){};
        \draw (3.5, -1.25) node[yshift=0.3cm] {$v_2$};
    \draw (2.5, -.45) node[vtx,scale=0.75] (v3){};
        \draw (2.5, -.45) node[yshift=-0.3cm] {$v_3$};
    \draw (1.5, -1.25) node[vtx,scale=0.75] (v4){};
        \draw (1.5, -1.25) node[yshift=0.3cm] {$v_4$};
    
    \draw [->,dashed] (4.5, -.45) -- (3.5, -.45);
        \draw [->,dashed] (3.5, -1.25) -- (2.5, -1.25);
         \draw [->,dashed] (2.5, -.45) -- (1.5, -.45);
         \draw [->,dashed] (1.5, -1.25) -- (0.5, -1.25);
         
    \draw [->,thick] (2.5,-1.75) -- (2.5, -2);
    
    \foreach \x in {0,...,4} {
       \draw (\x,-3.5) rectangle (\x+1,-2.2);
    }
    
    \draw (3.5, -2.45) node[vtx,scale=0.75] (v1'){};
        \draw (3.5, -2.45) node[yshift=-0.3cm] {$v_1$};
    \draw (2.5, -3.25) node[vtx,scale=0.75] (v2'){};
        \draw (2.5, -3.25) node[yshift=0.3cm] {$v_2$};
    \draw (1.5, -2.45) node[vtx,scale=0.75] (v3'){};
        \draw (1.5, -2.45) node[yshift=-0.3cm] {$v_3$};
    \draw (.5, -3.25) node[vtx,scale=0.75] (v4'){};
        \draw (.5, -3.25) node[yshift=0.3cm] {$v_4$};
    \draw (.5, -4.25) node{ };
    \end{tikzpicture}
    \caption{Shifting witnesses along $P$}
    \label{fig:shifting-witnesses}
    \end{minipage}
    
\end{figure}

Let $P=\gamma_{1}\dots\gamma_{t}$ be a path in $H$, where each arc $\gamma_{i}\gamma_{i+1}$
is witnessed by a vertex $v_{i}$.  
Obtain a 
 new $L$-coloring  $f'$ by
moving each $v_{i}$ from $\widetilde{\gamma}_{i}$ to $\widetilde{\gamma}_{i+1}$ (See Figure~\ref{fig:shifting-witnesses}).
Then $\wt \gamma_{1}(f')=\wt \gamma_{1}-v_{1}$, $\wt \gamma_{t}(f')=\wt \gamma_{t}+v_{t-1}$, $\wt \gamma_{i}(f')=\wt \gamma_{i}+v_{i-1}-v_{i}$
for $i\in[t-1]-1$ and $\widetilde{\beta}(f')=\widetilde{\beta}$ for all $\beta\in \Gamma\smallsetminus V(P)$.
We say that $f'$ is obtained by 
\emph{shifting witnesses along} $P$.

Suppose $p$ is movable to an accessible class $\wt\lambda_1$. 
Let $P=\lambda_1\dots\lambda_t$ witness that $\lambda_1$ is 
accessible from a light class $\lambda_t$. Obtain a new
$L$-coloring  $f'$ by moving $p$ to $\lambda_1$ and 
switching witnesses along $P$. 
We claim that $f'$ is an SE $L$-coloring of $G$. 
This operation only changes the size  of the light
class $\wt\lambda_t$: $|\wt\lambda_t(f')|=|\wt\lambda_t|+1$. 
Thus no classes of $f'$ are overfull, and so $f'$ is equitable.
To check that $f'$ does not have too many full classes, 
consider two cases. 
If $|G'|\bmod r\ne0$ then $|G|\mm r=(|G'|\mm r)+1$,
and we are allowed one additional full class. Else $|G'| \bmod r=0$; now $\lceil\frac{|G|}{r}\rceil=\lceil\frac{|G'|}{r}\rceil+1$ and $|G|\mm r=1$. Thus $\wt\lambda(f')$ is the only possible full class of $f'$, and one such class is allowed.

For the rest of the proof, assume that we must work harder, i.e.,
\begin{equation}
\mbox{$p$ is not movable to any accessible class.}\label{pnotmove}    
\end{equation}

\begin{lem}
\label{pr-1}
\begin{romaninline}
\item \label{pri}If $y\in{B}+p$ then $|L(y)\cap \Lambda|\geq r-b=a;$ 
\item \label{prii} if $y\in B$ then $y$ is not movable to any accessible class; and 
\item \label{priii}$a\leq d(p)$. 
\end{romaninline}
\end{lem}

\begin{proof}\ref{pri}
 Observe that $|L(y)\cap \Lambda|=|L(y)\smallsetminus \Phi|\geq r-b=a$. 

\ref{prii} 
Suppose $y\in \wt\phi$ with $\phi\in\Phi$ and $y$ is movable to an accessible class $\wt\alpha_1$. 
Let $P=\alpha_1\dots \alpha_t$ witness that $\alpha$ is accessible. Now 
 $\phi\alpha P$ 
witnesses that $\beta\in \Lambda$, a contradiction.

By \eqref{pnotmove}, $p$ has a neighbor in every class $\wt\lambda$ with $\lambda\in \Lambda\cap L(p)$, so by \ref{pri}, $a\le d(p)$. 
\end{proof}

We end this section by establishing the properties of graphs in $\BB$ that we will need for our proof.

\begin{lem}\label{lem:BB}
    $G$  satisfies:
    \begin{enumerate}
        \item $K_{3,3}\nsubseteq G$;
        \item $\delta(G)\le6$; and
        \item if $\Theta\subseteq \Gamma$ is a sink then
\begin{romaninline}
\item \label{|sink|i} $|\Theta|\leq2$ or $|\Theta|\geq r-2$ and
\item \label{|sink|ii} $b\geq r-2$ and $a\le 2$.
\end{romaninline}
    \end{enumerate}
\end{lem}

\begin{proof}
    As $G\in \BB$, $G$ satisfies \eqref{eq:bibound}, 
    so $K_{3,3}\nsubseteq G$. For (b), consider a spanning bipartite 
    subgraph $G_0$ with $\|G_0\|$ maximum. By \eqref{eq:bibound}, 
    $\delta(G_0)\le3$. Say $d_{G_0}(x)\le3$. Then $d_G(x)\le6$ 
    since otherwise moving $x$ to the other part of $G_0$ would 
    increase $\|G_0\|$.

For (c), put $\Theta'=\Gamma\smallsetminus \Theta$. As $\Theta$ is a sink, if vertex $v\in\widetilde{\Theta}$ is not movable to any class
in $\widehat{\Theta}'$, and so $\|v,\widetilde{\Theta}'\|\geq|L(v)\smallsetminus \Theta|\geq r-|\Theta|$.
By (\ref{eq:bibound}),
\begin{align*}
2rs>2|\widetilde{\Theta}\cup\widetilde{\Theta}'|>\|\widetilde{\Theta},\widetilde{\Theta}'\| & \geq(r-|\Theta|)|\Theta|s.
\end{align*}
As $r\geq9$, $|\Theta|\leq2~\text{or}~|\Theta|\geq r-2$, so (ci) is proved.
By Lemma~\ref{pr-1}(iii) and (b), $|\Phi|=r-a\geq 9-6\geq3$. Since
$\Phi$ is a sink in $H$, (ci) gives  $|\Phi|\geq r-2$   and $a\le 2$. 
\end{proof}

\section{Proof of Theorem~\ref{thm:Main2}}

Choose an  SE $L$-coloring $f$ of $G'$  so that 
\begin{Oinline}
\item \label{enu:1}$A=\widetilde{\Lambda}(f)$ is maximum, and subject to this, 
\item \label{enu:2}$\widehat{\Lambda}(f)$ has as many nonempty classes as possible. 
\end{Oinline}

\begin{lem}\label{lem:Free}
\begin{romaninline}
\item \label{enu:eset}
If there is $\lambda\in \Lambda$ with $\wt \lambda =\emptyset$ then there is $\lambda'\in\Lambda$ with $|\wt \lambda'|=1$.  \\
\item \label{enu:light}
If $|\Lambda|>a$ then every class in $\widehat{\Lambda}$ is light. \\
\item \label{enu:plain} If  $|\Lambda|=a$ then there is at most one full class  in $|\wt{\Lambda}|$, and only when $a=2$.\\
\end{romaninline}
\end{lem}

\begin{proof}
\ref{enu:eset} Suppose $\lambda\in \Lambda$ with $\wt \lambda =\emptyset$. Now there is $v\in V$ with $\lambda\in L(v)$, and $v$ is movable to $\wt \lambda=\emptyset$. By \eqref{pnotmove} and Lemma~\ref{pr-1}\ref{prii}, $v\in A$. Moving $v$ to $\wt \lambda$  contradicts \ref{enu:2} unless $|\wt v|=1$.

 By Lemma~\ref{lem:BB}(cii), $a\le2$. Now $|A|\le as-1$. If $a=1$ then there are no full classes. Else $a=2$, and  
 $|\Lambda|\geq3$.
 Let $\widetilde{\lambda}_{1}$, $\widetilde{\lambda}_{2}$,
$\widetilde{\lambda}_{3}$ be the three largest classes in $\wh\Lambda$, arranged in non-increcreasing order. Suppose $|\wt\lambda_{1}|=s$. By \ref{enu:eset},
if $\widetilde{\alpha}_{3}=\emptyset$ then 
$|\wt\lambda_{2}|=1$; else $|\wt\lambda_{2}|\le2s-1-|\wt\lambda_{1}|-|\wt\lambda_{3}|\le s-2$. Anyway $|\wt\lambda_{i}|\le|\wt\lambda_{2}|\le s-2$ for all $i\ge 2$.
As  $\wt\lambda_{1}$ is accessible, there is $v\in \wt\lambda_{1}$ with $v$ movable to some class of size at most $s-2$. This contradicts \ref{enu:1}.

\ref{enu:plain}  As  $|A|\le as-1$ and $a\le 2$,  there are only enough vertices for one full class, and only when $a=2$.
\end{proof}

For $z\in A$ and $y\in B$, call $y$ a \emph{solo neighbor}
of $z$ if $z$ blocks $y$, \emph{i.e.,} $f(z)\in L(y)$ and $N(y)\cap\widetilde{z}=\{z\}$; 
call $z$ a \emph{solo vertex} if it has a solo neighbor, and 
additionally, if there is a full class in
$\Lambda$, then it is $\wt z$. 
Let $S_{z}$ be the 
set of solo neighbors of $z$. Call $y\in S_{z}$ \emph{useful} if $yy'\notin E$
for some $y'\in S_{z}-y$, and let $S_{z}^{*}$ be the set of useful $y\in S_{z}$. 

From an algorithmic point of view, Lemmas~\ref{lem:soloToA}, \ref{lem:Useful} and \ref{lem:S*capD} below provide methods for improving the SE $L$-coloring $f$. As we are assuming $f$ is optimal, they are phased negatively.

\begin{lem}
\label{lem:soloToA}If $z$ is a solo vertex  then
\begin{enumerate}
\item \label{lta1}$z$ is
not movable to any class in $\widehat{\Lambda}$;
\item \label{lta2} $\|z,A\|\geq|L(z)\cap \Lambda|-1\geq a-1$;
\item \label{lta3} $\|z,B\|\leq |L(z)\cap\Phi|+1\le b+1$; and 
\item \label{lta4} if $\|z,B\|=b+1$ then $\Phi\subseteq L(z)$.
\end{enumerate}
\end{lem}

\begin{proof}
Suppose $z$ movable to $\widetilde{\lambda}\in \wh \Lambda$ and $y\in S_z$.
Obtain a new SE $L$-coloring $f'$ by moving $z$ to $\widetilde{\lambda}$ and $y$ to $\wt z$. 
 Now $\wt{\lambda}+z$ is not overfull because $z$ is solo, 
 and so  $\wt z$ is the only possibly full class in 
 $\wt{\Lambda}$. Also, since $\wt y-y$ is now light, the 
 number of full classes has not increased. 
 By \ref{enu:1}, $\wt\lambda(f')$ is full. 
 Thus $a=2$. By \ref{enu:1} and the definition of 
 $\Lambda$, $\Phi-f(y)$ is a sink, contradicting Lemma~\ref{lem:BB}(ci). 
 So \eqref{lta1} holds, and
\begin{align*}
\|z,A\|&\geq|(L(z)\cap \Lambda)-f(z)|\geq a-1~\text{and}\\
\|z,B\|&\leq r-|L(z)\cap\Lambda|+1
=|L(z)\cap\Phi|+1\le b+1.
\end{align*}
Thus \eqref{lta2} and \eqref{lta3} hold. If $\|z,B\|=b+1$, then $|L(z)\cap \Phi|=r-|L(z)\cap \Lambda|=b$, so $\Phi\subseteq L(z)$.
\end{proof}
\begin{lem}
    \label{Af'} Let $z$ be a solo vertex with $y\in S^*_z$ and  $f(z)=\lambda$. If  $f'$ is an equitable $L$-coloring such that $\wt{\gamma}(f')=\wt{\gamma}$
    for all $\gamma\in \Lambda-\lambda$ and 
    $\wt{\lambda}(f')=\wt{\lambda}+ y-z$ then $\wt{\Lambda}+y-z\subseteq \wt{\Lambda}(f')$.
 \end{lem}

 \begin{proof}
Since $z$ is solo, every class  other than $\wt z$ is light, and so accessible.
 If $\wt z$ is full then there is a witness $w\in \wt \lambda$ with $w$ movable to some light class. By Lemma~\ref{lem:soloToA}\eqref{lta1},  $w\ne z$. Thus  $w\in \wt{\lambda}(f')$, and so $\wt{\lambda}(f')$ is accessible.
\end{proof}
\begin{lem}
\label{lem:Useful}If $z$ is a solo vertex, $y\in S_{z}^{*}$ and $f(y)\in L(z)$ then $\|z,\widetilde{y}\|\geq2$.
\end{lem}

\begin{proof}
Otherwise, $N(z)\cap\widetilde{y}=\{y\}$. Obtain an SE $L$-coloring $f'$
from $f$ by moving $z$ to $\widetilde{y}$ and $y$ to $\widetilde{z}$. By Lemma~\ref{Af'} $A\cup\widetilde{y}'+y-z\subseteq\widetilde{\Lambda}(f')$,
where $y'\in S_{z}^{*}$ with $yy'\notin E$, contradicting \ref{enu:1}. 
\end{proof}

\begin{lem}
\label{lem:S*capD}Suppose $z$ is a solo vertex, $y\in S_{z}^{*}$ and $z$ is movable
to $\widetilde{\beta}\in\widehat{\Phi}$. Then $f(y)$ is not reachable from $\beta$.
\end{lem} 

\begin{proof}
Otherwise there is a $\beta,f(y)$-path $P$ in $H$. Obtain a
new SE $L$-coloring $f'$ by moving $z$ to $\widetilde{\beta}$, switching
witnesses along $P$ and moving $y$ to $\widetilde{z}$. 
By Lemma~\ref{Af'},
 $A+y-z\subseteq\widetilde{\Lambda}(f')$. But
$y$ is nonadjacent to some $y'\in S_{z}^{*}$. So $\widetilde{y}'(f')\subseteq\widetilde{\Lambda}(f')$,
contradicting \ref{enu:1}. 
\end{proof}

Call $f$ \emph{extreme} if there is a solo vertex $z_0$ 
and there are 
 subsets $\Theta_{0},\Upsilon,\Upsilon'\subseteq \Phi$ 
such that
\begin{enumerate}[label=(E\arabic{enumi})]
\item \label{enu:E0}
if $y\in N(z_0)\cap B$ then $f(y)\in L(z)$; 
\item \label{enu:E2}$\emptyset\ne S^*_{z_0}\subseteq\wt\Upsilon$ and $|S_{z_0}^{*}\cap\widetilde{\upsilon}|=2$
for all $\upsilon\in\Upsilon$;
\item \label{enu:E3}$N(z_0)\cap B\smallsetminus(S_{z_0}^{*}\cup\widetilde{\Theta}_{0})\subseteq\widetilde{\Upsilon}'$
and $\|z_0,\widetilde{\phi}\|\leq1$ for all $\phi\in\Upsilon'$; and
\item \label{enu:E1}$|\Theta_{0}|\leq2$ and $\Upsilon$ is reachable from every 
 $\phi\in\Phi\smallsetminus\Theta_{0}$.
\end{enumerate}

The following Lemma will be used to extend $f$ to an SE $L$-coloring of $G$.

\begin{lem}\label{lem:ext}
If $f$ is extreme then $G$ has an SE $L$-coloring.
\end{lem}

\begin{proof}
By Lemma~\ref{lem:BB}(b), there are at least $r-d(p)\geq3$ classes to which $p$ can be moved. By \eqref{pnotmove},
none of them are in $\widehat{\Lambda}$. By \ref{enu:E1}, $p$ can be moved to some class
$\phi_{1}\in\Phi\smallsetminus\Theta_{0}$, and there is a $\phi_{1},\Upsilon$-path
$P=\phi_{1}\dots\phi_{h}$; say $\upsilon:=\phi_{h}$. For each $i\in[h-1],$ let
$w_{i}$ witness $\phi_{i-1}\phi_{i}$, and put $w_{0}:=p$. By \ref{enu:E2}, there 
are distinct $y,y'$ with $\{y,y'\}= S_{z_0}^{*}\cap\widetilde{\upsilon}$, see Figure~\ref{fig:case1}.
Put $\alpha:=f(z_0)$.

First suppose $w_{h-1}\notin N(z_0)$. Obtain an $L$-coloring $f'$ of $G$ by moving $p$ to $\phi_1$, switching
witnesses along $P$, moving $z_0$ to $\widetilde{\upsilon}+w_{h-1}-y$ and $y$ to $\widetilde{\alpha}-z_0$;
here we use \ref{enu:E0}. Using \ref{enu:E2},
this preserves
the size and independence of all classes, 
except that $z_0,y'\in\wt \upsilon$ and $|\wt{\upsilon}(f')|=|\wt{\upsilon}|+1$. 
Also, $y'$ witnesses that $\upsilon\alpha$ is an edge in $H(f')$. 
By Lemma~\ref{Af'}, $\widetilde{\alpha}(f')$ is accessible. Let $P$ be a path in $H(f')$  witnessing this. Switching witnesses along 
$\upsilon\alpha P$ yields an equitable coloring of $G$.

Otherwise $w_{h-1}\in N(z_0)$. By \ref{enu:E3}, 
$N(z_0)\cap\widehat{\phi}_{h-1}=\{w_{h-1}\}$.
Obtain a new $L$-coloring $f'$ by 
putting $z_0$ in $\widetilde{\phi}_{h-1}-w_{h-1}$,
$w_{h-1}$ in $\widetilde{\upsilon}-y$ and $y$ in $\widetilde{\alpha}-z_0$; again
we use \ref{enu:E0} to check that $\phi_{h-1}\in L(z_0)$. This operation preserves the 
size and independence of all color classes. 
Now 
$A+y-z_0\subseteq A(f')$ by Lemma~\ref{Af'}, and  
$y'$ is movable to $\wt{\alpha}(f')$. Thus
 $A\cup \wt{\upsilon}+y-z_0\subseteq A(f')$, contradicting \ref{enu:1}.  
\end{proof}

Finally, the last three lemmas of this section are used to show that $G$ has a solo vertex with enough solo neighbors and useful solo neighbors to gainfully apply the previous lemmas on improving $f$.

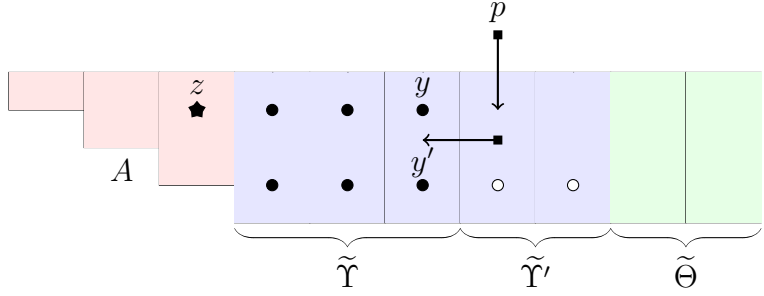
\begin{figure}[t]
    \centering
        \begin{tikzpicture}

    \draw (0,-.5) rectangle (1,0);
       \fill[red!10!white] (0,-.5) rectangle (1,0);
    \draw (1,-1) rectangle (2,0);
       \fill[red!10!white] (1,-1) rectangle (2,0);
    \draw (2,-1.5) rectangle (3,0);
       \fill[red!10!white] (2,-1.5) rectangle (3,0);
    \draw (2.5,-.5) node[vtx,scale=0.75,shape=star] (z){};
     \draw (2.5,-.5) node[yshift=0.3cm] {$z$};
     \draw (1.5,-1.3) node(A) {$A$};

     \foreach \x in {1,...,5} {
       \draw (\x+2,-2) rectangle (\x+3,0);
       \fill[blue!10!white] (\x+2,-2) rectangle (\x+3,0);
       \draw (\x+2.5, 0) node[vtx,scale=00] (\x){};
    }
    
    \draw (3.5,-1.5) node[vtx,scale=0.75] (y1){};
     \draw (3.5,-.5) node[vtx,scale=0.75] (y2){};
     
     \draw (4.5,-1.5) node[vtx,scale=0.75] (y3){};
     \draw (4.5,-.5) node[vtx,scale=0.75] (y4){};
     
     \draw (5.5,-1.5) node[vtx,scale=0.75] (y5){};
     \draw(5.5,-1.5)  node[yshift=0.3cm] {$y'$};
     \draw (5.5,-.5) node[vtx,scale=0.75] (y6){};
     \draw(5.5,-.5)  node[yshift=0.3cm] {$y$};
    
     \draw (6.5,-1.5) node[vtx, scale=0.75, shape=circle,fill=white] (y7){};

    
     \draw (7.5,-1.5) node[vtx, scale=0.75, shape=circle,fill=white] (y8){};
    
     \foreach \x in {8,9} {
       \draw (\x,-2) rectangle (\x+1,0);
       \fill[green!10!white] (\x,-2) rectangle (\x+1,0);
    }

    \draw (6.5, .5) node[vtx, scale=0.75, shape=rectangle,fill=black] (p){};
    \draw (6.5, .5)  node[yshift=0.3cm] {$p$};
    \draw [->,thick] (6.5, .5)  -- (6.5, -.5);
    \draw (6.5, -.9) node[vtx, scale=0.75, shape=rectangle,fill=black] (w){};
    \draw [->,thick] (6.5, -.9)  -- (5.5, -.9);
    
    \draw [decorate,
    decoration = {calligraphic brace,
        raise=5pt,
        aspect=0.5,
        amplitude=7pt}] (10,-1.9) -- (8,-1.9);
    
        \draw(9,-2.65) node {$\widetilde{\Theta}$};
        \draw [decorate,
    decoration = {calligraphic brace,
        raise=5pt,
        aspect=0.5,
        amplitude=7pt}] (8,-1.9) -- (6,-1.9);
    
        \draw(7,-2.65) node {$\widetilde{\Upsilon'}$};
        \draw [decorate,
    decoration = {calligraphic brace,
        raise=5pt,
        aspect=0.5,
        amplitude=7pt}] (6,-1.9) -- (3,-1.9);
    
        \draw(4.5,-2.65) node {$\widetilde{\Upsilon}$};
    \end{tikzpicture}
    \caption{Circles depict neighbors of $z$; solid circles depict useful solo neighbors. 
    }
    \label{fig:case1}
\end{figure}

Consider a function $\mu:\Phi\rightarrow[0,1]$, and set $\mu(\Phi)=\sum_{\beta\in \Phi}\mu(\beta)$.
Extend $\mu$ to $B=\widetilde{\Phi}$ by $\mu(y)=\mu\circ f(y)$. 
Now, define a
weight function $w: A\times B\rightarrow\mathbb{R}^{+}$ in
terms of $\mu$ by 
\begin{align*}
w(x,y) & =\begin{cases}
\frac{\mu(y)}{\|y,\widetilde{x}\|}, & \text{if }f(x)\in L(y)\text{ and }xy\in E\\
0, & \text{else}
\end{cases}.
\end{align*}
Finally, put $w(x)=\sum_{y\in B}w(x,y)$. 

\begin{lem}
\label{lem:w(z)}For every choice of $\mu:\Phi\rightarrow[0,1]$ \textup{there is} $\ensuremath{z\in A}$
with
\begin{equation}
w(z)=\ensuremath{\sum_{y\in B}w(z,y)>\mu(\Phi)}.\label{eq:solo}
\end{equation}
Moreover, if $L$ is plain then for all   $\lambda\in\Lambda$  either 
there is $z\in\wt{\lambda}$ satisfying \eqref{eq:solo} or both $\wt{\lambda}$ is full and $w(z)=\mu(\Phi)$ for every $z\in \widetilde{\lambda}$.
\end{lem}

\begin{proof}
For each $y\in B$ with $f(y)=:\phi$, 
\begin{align*}
\sum_{x\in A}w(x,y) & =
\sum_{\alpha\in L(y)\cap \Lambda}~
\sum_{x\in\widetilde{\alpha}}w(x,y)=
\sum_{\alpha\in L(y)\cap \Lambda}~
\sum_{x\in\widetilde{\alpha}\cap N(y)}
\frac{\mu(\phi)}{\|y,\widetilde{\alpha}\|}=
\sum_{\alpha\in L(y)\cap \Lambda}\mu(\phi)
\geq a\mu(\phi).
\end{align*} 
Thus 
\[
\sum_{x\in A}\sum_{y\in B}w(x,y)=\sum_{y\in B}\sum_{x\in A}w(x,y)\geq\sum_{y\in B}a\mu(f(y))=as\mu(\Phi).
\]
As $|A|\leq as-1$, there is $\ensuremath{z\in A}$ satisfying
(\ref{eq:solo}).

Finally, suppose $L$ is plain and  $\lambda\in\Lambda$. Now $\lambda\in L(y)$ for all $y\in V$. So for each $y\in B$, 
\begin{align*}
\sum_{x\in\widetilde{\lambda}}w(x,y)&=
\sum_{x\in\widetilde{\lambda}\cap N(y)}
\frac{\mu(f(y))}{\|y,\widetilde{\lambda}\|}=
\mu(\lambda).
\end{align*}
Thus
\[
\sum_{x\in \wt{\lambda}}\sum_{y\in B}w(x,y)=
\sum_{y\in B}\sum_{x\in \wt{\lambda}}w(x,y)\geq
\sum_{y\in B}\mu(f(y))=s\mu(\Phi).
\]

If $\wt{\lambda}$ is light then $|\lambda|\le s-1$, so there is  $z\in \widetilde{\lambda}$ satisfying \eqref{eq:solo}. Else $|\wt{\lambda}|=s$. Then either 
there is $z\in\wt{\lambda}$ satisfying \eqref{eq:solo} or both $\wt{\lambda}$ is full and $w(z)=\mu(\phi)$ for every $z\in \widetilde{\lambda}$ .
\end{proof}
\begin{lem}
\label{lem:manySolo}
Suppose $\mu:\Phi\to\{1/2,1\}$  with $\mu(\Phi)\geq b-g$, where $g\in\{0,1/2\}$, and set $$\Theta:=\{\phi\in \Phi:\mu(\phi)=1\}.$$ Then there is a solo vertex z.
If $g=0$ then 
\begin{romaninline}
\item\label{soloi}$|S_{z}|\geq b$, 
\item \label{soloii}$\|z,B\|=b+1$ and
\item \label{soloiii}$\Phi\subseteq L(z)$;
if $g=1/2$ then
\item \label{soloiv}
 $|S_{z}\cap\wt{\Theta}|+\|z,B\|\geq2b$ and
\item \label{solov}
$b-1\leq|S_{z}\cap\wt{\Theta}|\leq\|z,B\|\leq b+1$.
\end{romaninline}
\end{lem}

\begin{proof}
Choose $z\in A$ so that $w(z)$ is maximum
subject to the condition that if there is 
$\lambda\in\Lambda$ with $\widetilde{\lambda}$
full then $z\in\widetilde{\Lambda}.$ 
By Lemma~\ref{lem:w(z)}, $w(z)\geq \mu(\Phi)$ and if $w(z)=\mu(\Phi)$
then $\widetilde{z}$ is full and $w(x)=\mu(\Phi)$ for all $x\in\widetilde{z}$. For any
$x\in A$ and $y\in B$,  $w(x,y)=1$ 
if and only if   $y\in S_{x}\cap \wt{\Theta}$. Also if $w(x,y)\ne1$
then $w(x,y)\leq1/2$. Thus 
\begin{equation}
w(z)\leq|S_{z}|+(\|z,B\|-|S_{z}|)/2.\label{eq:=000023solo}
\end{equation}
As $w(x)\ge \mu(\Phi)\geq b-g,$ and $\|z,B\|\leq b+1$ 
by Lemma~\ref{lem:soloToA}\eqref{lta3}, \eqref{eq:=000023solo} yields
$5\leq b-1-2g\leq|S_{z}|$. Thus $z$ is solo. Moreover, if $w(z)=\mu(\Phi)$ then $\widetilde{z}$
is full and every $x\in\widetilde{z}$ is solo. But in this case, some $x\in\widetilde{z}$
is movable to a class in $\widehat{\Lambda}$, contradicting Lemma~\ref{lem:soloToA}\eqref{lta1}. Thus $w(z)>\mu(\Phi)= b-g$.

Suppose g=0. Then  $\|z,B\|>b$, and so \ref{soloii} $\|z,B\|=b+1$. Now \eqref{eq:=000023solo}
yields \ref{soloi} $|S_{z}|\geq b$, and by Lemma~\ref{lem:soloToA}\eqref{lta4}, \ref{soloiii} $\Phi\subseteq L(z)$. Else g=1/2. Using  $w(z)> b-1/2$, we have $b\le\|z,B\|\le b+1$. In the first case
  \ref{soloiv} holds, and in the second case \ref{solov} holds.
\end{proof}

\begin{lem}
\label{lem:S^*_z}If $z\in A$ and $|S_{z}|\geq5$ then $|S_{z}^{*}|\geq|S_{z}|-1$.
\end{lem}

\begin{proof}
Otherwise there are 
 distinct $y,y'\in S_{z}$ that are not useful. Then $S\cup\{y,y',z\}$ induces a supergraph of $K_{3,3}$, contradicting Lemma~\ref{lem:BB}(a). 
\end{proof}

\subsection*{Finishing the proof of Theorem~\ref{thm:Main2}}
We will consider two cases. 
In each case, we will show that for some $z_0,\Theta_0$ 
the string $(z_0,\Theta_0,\Upsilon,\Upsilon')$ 
witnesses that $f$ is extreme, 
where
\[
\Upsilon:=\{\phi\in\Phi:\|z_0,\widetilde{\phi}\cap 
S_{z_0}^{*}\|\geq1\}~\text{and}~\Upsilon':=
\{\phi\in\Phi\smallsetminus\Upsilon:
\|z_0,\widetilde{\phi}\|\geq1\}.
\]
Then by 
Lemma~\ref{lem:ext} $G$ will have an SE $L$-coloring.

In both cases $(z_0,\Theta_0)$ will satisfy:
\begin{Cenum}
\item \label{g1}
$\Theta_0$ is a sink with 
$|\Theta_0|\le2$, ~
\item \label{g2}
$z_0$ is a solo vertex with $|S_{z_0}|\ge7$, and ~ 
\item \label{g3}
if $z_0$ is movable to $\wt{\phi}\in\wh\Phi$ then $\phi\in \Theta_0$.
\end{Cenum}

By \ref{g3}, 
\begin{romaninline}
\item \label{enu:A}$\Phi\cap L(z_0) \subseteq\Theta_0\cup\Upsilon\cup\Upsilon'\subseteq \Phi$. By Lemma~\ref{lem:Useful}, 
\item \label{enu:B} if $\upsilon\in\Upsilon$ then $2\leq\|z_0,\widetilde{\upsilon}\|$.
Also 
\item \label{enu:C}if $\upsilon\in\Upsilon'$ then $1\leq\|z_0,\widetilde{\upsilon}\smallsetminus S_{z_0}^{*}\|$. 
 By Lemma~\ref{lem:S^*_z}, \item \label{enu:D}
 $|\Upsilon'|\le \|z_0,B\|-|S^*_{z_0}|\le\\ \|z_0,B\|-|S_{z_0}|+1$.
\end{romaninline}
Using Lemma~\ref{lem:soloToA}(c) and combining
\ref{enu:A}, \ref{enu:B} and \ref{enu:C}, 
\begin{align}\label{eq:main}
\|z_0,B\|-1-|\Theta_0|+|\Upsilon|&\le
|\Phi\cap L(z_0)|-|\Theta_0|+|\Upsilon| \notag\\ &\leq2|\Upsilon|+|\Upsilon'|\leq
\sum_{\upsilon\in\Upsilon\cup\Upsilon'}\|z_0,\widetilde{v}\|
\leq\|z_0,B\|.
\end{align}
 Hence
 \begin{romaninline}[resume]
    \item \label{enu:E}
    $|\Upsilon|\le|\Theta_0|+1$.
\end{romaninline}
By \ref{enu:A}, \ref{enu:D} and (C1,2):
\begin{align}\label{S_z}
    \|z_0,B\|-1&\le |\Phi\cap L(z)|\le|\Theta_0|+|\Upsilon|+|\Upsilon'|
    \le2+3+\|z_0,B\|-|S_{z_0}|+1\\\label{S_z2}
     7\le|S_{z_0}|&\le7.\notag
\end{align}
Now equality holds throughout  \eqref{S_z}. So $|\Theta_0|=2$ and $|\Upsilon|=3$. Thus equality holds throughout \eqref{eq:main}. So \ref{enu:E2} and \ref{enu:E3} will hold.
For \ref{enu:E1}, consider $\phi\in\Phi\smallsetminus\Theta_0$. By Lemma~\ref{lem:S*capD}, $\Upsilon\cap\Theta_0=\emptyset$.
Set $\Psi':=\Theta_0+\phi$ and $\Theta':=\cl^{+}(\Psi')$. As 
 $\Theta'$ 
is a sink with 
 $|\Theta'|\geq3$, Lemma~\ref{lem:BB}(ci) 
 implies $|\Theta'|\ge r-2$. As $|\wt{\Upsilon}|=3$, 
 there is 
 $\upsilon\in\Upsilon\smallsetminus\Theta_0$ that is reachable from $\Theta'$. As $\Theta_0$ is a sink, $\upsilon$ is not 
  reachable from $\Theta_0$. So $\upsilon$ is reachable from $\phi$. 

Now it remains to choose $(z_0,\Theta_0)$ satisfying (C1--C3) and to check that \ref{enu:E0}  holds for this choice. By Lemma~\ref{lem:manySolo}, there is a solo vertex $z$ with $|S_{z}|\geq b$, $\|z,B\|=b+1$
and $\Phi\subseteq L(z)$. Let $\Psi$ be the set of $\beta\in \Phi$ with $z$ movable to
$\widetilde{\beta}$, and put $\Theta:=\cl^{+}(\Psi)$. 
By Lemma~\ref{lem:BB}(ci), $|\Theta|\leq2$ or  
$|\Theta|\geq r-2$.

\medskip\noindent
\emph{Case 1:} $|\Theta|\leq2$. Set $z_0:=z$ and 
$\Theta_0:=\Theta$. By the case,  \ref{g1} holds. 
 As $|S_z|\ge b\ge7$, \ref{g2} holds.  
By the definition of $\Theta$, \ref{g3} holds.
As $\Phi\subseteq L(z)$,  \ref{enu:E0} holds.

\medskip\noindent
\emph{Case 2: $|\Theta|\geq r-2$}.  
By Lemma~\ref{lem:S*capD},
$S_{z}^{*}\subseteq B\smallsetminus\wt\Theta$,  so $|\Phi\smallsetminus \Theta|\geq1$ and $s\ge6$. 
Thus $|\Theta|=r-2,$ $|\Phi\smallsetminus \Theta|=1$, $b=r-1$ and  $a=1$. Put $\Phi\smallsetminus \Theta=:\{\gamma_{0}\}$, and set
\begin{align*}
\mu(\beta):=\begin{cases}
1, & \text{if }\beta\in \Theta\\
1/2, & \text{if }\beta=\gamma_{0}
\end{cases}.
\end{align*}
By Lemma~\ref{lem:manySolo}
(iv,v), there is a solo vertex $z'$ with $w(z')> b-1/2$ and
\begin{equation}\label{eq:z'}
    |S_{z'}\cap\Theta|+\|z',B\|\geq2b~ \text{ and }~
b-1\leq|S_{z'}\cap\Theta|\leq\|z',B\|\leq b+1. 
\end{equation}
Put $\Psi':=\{\beta\in \Phi: z' \textrm{ movable to }
\widetilde{\beta}\}$ and $\Theta':=\cl^{+}(\Psi')$. Now set $z_0:=z'$ and $\Theta_0:=\Theta'$.
Clearly (C1--C3) hold. 

We claim that $\gamma_0\notin \Psi'$. Else move $z'$ to $\widetilde{\gamma}_{0}$\,
($\widetilde{\gamma}_{0}+z'$ is  overfull). As in the proof of Lemma~\ref{lem:Useful}, since $K_{3,3}\nsubseteq G$, 
there is at most one $y_{0}'\in S_{z'}^{*}\cap \wt\Theta$ with 
 $\|y'_{0},S^{*}_{z}\|\ge|S^{*}_{z}|-1$. Thus 
there are $y'\in S_{z'}^{*}\cap \wt\Theta$ and $y_{1},y_{2}\in  S_{z}^{*}$ 
with $y_{1},y_{2}\notin N(y')$.
Move $y'$ to $\widetilde{z}'-z'$ \,($\widetilde{y}'-y'$ is light). Let $P=\beta_{1}\dots\beta_{k}$
be a $\Psi,f(y')$-path in $H$. Move $z$ to $\widetilde{\beta}_{1}$ and shift
witnesses along $P$ \,(now all classes in $\widehat{\Theta}$ are full). Finally, move
$y_1$ to $\widetilde{z}$ \,($\widetilde{\gamma}_{0}+z'-y_{1}$ is full). 
 Since $y_{2}$ is movable to $\wt z-z+y_{1}$  (even if $\wt z=\wt z'$),
this yields a new
SE $L$-coloring $f'$ with 
$A\cup\widetilde{\gamma}_{0}-z\subseteq\widetilde{\Lambda}(f')$,
contradicting \ref{enu:1}.

As $\gamma_0\notin\Psi'$,  
either $\gamma_{0}\notin L(z')$ or 
$\|z',\widetilde{\gamma}_0\|\geq1$. 
If $\gamma_{0}\notin L(z')$ 
then $|L(z')\cap\Phi|\le b-1$, so 
$\|z',B\|\le b$ by Lemma~\ref{lem:soloToA}\eqref{lta3}.  
By \eqref{eq:z'}, 
$|S_{z'}\cap\wt{\Theta}|=b$, 
so $N(z')\cap B\subseteq \wt{\Theta}$ and 
$|L(z')\cap\Phi|\ge b-1$.
 Thus $L(z')\cap \Phi=\Theta$, and 
\ref{enu:E0} holds.
Else there is $y\in N(z')\cap\widetilde{\gamma}_0$.
Since $w(z',y)\le 1/2$ 
and $w(z')\ge b$, $\|z',B\|= b+1$. 
Thus $\Phi\subseteq L(z')$, and \ref{enu:E0} holds.

We have shown that  some sequence, $(z_0,\Theta_0,\Upsilon,\Upsilon')$, witnesses that $f$ is extreme. 
By Lemma~\ref{lem:ext}, 
$G$ has an SE $L$-coloring, 
completing the proof of  Theorem~\ref{thm:Main2}.

\section{Concluding remarks} 

1. Our Theorem~\ref{thm:Main} establishes a statement implying the validity of both, Conjecture~\ref{conj:CLW}  and
Conjecture~\ref{conj:KPW},  for planar graphs with maximum degree at least $9$. This suggests that possibly the following is true.

\begin{conjecture}\label{conj:KKX}
If $G$ is an $r$-colorable graph with $\Delta(G)\leq r$ then either $r$ is odd and $K_{r,r}\subseteq G$ or
$G$ is equitably
$r$-choosable, and even SE $r$-choosable.  
\end{conjecture}

It might be interesting to determine whether Conjecture~\ref{conj:KKX}  (and Conjecture~\ref{conj:CLW}) holds for all planar graphs. Also one could check whether Conjecture~\ref{conj:KKX} holds in special cases (e.g. $r\le4$ or $r\ge n/4$) for which Conjecture~\ref{conj:CLW} is known to be true.  

2. Our Theorem~\ref{thm:Main2} deals with 
 class $\BB$ which topologically is much broader than the class of  planar graphs. For example, for any graph $H$, the graph $B_H$ obtained from $H$ by subdividing each edge once is in $\BB$. 
Therefore, every graph is a minor of a graph in $\BB$.
In particular, the acyclic chromatic number of graphs in $\BB$ can be arbitrarily large, while it is at most $5$ for any planar graph.

So, proving  
Conjecture~\ref{conj:KKX} for all graphs in $\BB$ would be more interesting than proving it merely for planar graphs.

Another way to extend results on equitable coloring of planar graphs is to consider graphs embeddable into other surfaces. For example, it was proved in~\cite{KLX} that for each $r\geq 9$ every graph $G$ embeddable into torus or the M\" obius strip with $\Delta(G)\leq r$ has an equitable $r$-coloring.

3. In~\cite{KK5}, the following refinement of the Hajnal-Szemer\' edi Theorem was proved.

\begin{thm}\label{thKK} For any $r\geq 1$, if for every edge $xy$ in a 
 graph $G$,  $d(x)+d(y)\geq 2r+1$, then $G$ 
 has an equitable $r+1$ coloring.
\end{thm}

It would be interesting to prove a list analogue of Theorem~\ref{thKK} for wide classes of graphs. Is this analogue true for planar graphs? For graphs in $\BB$?

4. Recall that one can extract from our proof of Theorem~\ref{thm:Main2} a polynomial-time algorithm
that produces for every
 graph $G\in \BB$  with $\Delta(G)\leq r$ and every $(r + 1)$-list $L$
 for $G$ an equitable $L$-coloring of $G$.
It is also known (see, e.g.~\cite{KKMS}) that there is a polynomial-time algorithm for finding an equitable $(r+1)$-coloring of an arbitrary graph with maximum degree at most $r$. But we do not know a similar algorithmic analogue of 
Theorem~\ref{thKK}.

\end{document}